\documentclass [12pt]{amsart}
\usepackage{amsmath,amssymb}
 \usepackage{amsthm, amsfonts}
 \usepackage{enumerate}

\newtheorem{theorem}{Theorem}[section]

\newtheorem{lemma}[theorem]{Lemma}
\newtheorem{corollary}[theorem]{Corollary}
\theoremstyle{definition}
\newtheorem{definition}[theorem]{Definition}
\newtheorem{example}[theorem]{Example}

\theoremstyle{remark}

\numberwithin{equation}{section}

\begin{document}

\title [{On graded $I_{e}$-prime submodules}]{On graded $I_{e}$-prime submodules of graded modules over graded commutative rings }

 \author[{{S. Alghueiri and K. Al-Zoubi }}]{\textit{Shatha Alghueiri and Khaldoun Al-Zoubi*  }}

\address
{\textit{Shatha Alghueiri, Department of Mathematics and Statistics,
Jordan University of Science and Technology, P.O.Box 3030, Irbid
22110, Jordan.}}
\bigskip
{\email{\textit{ghweiri64@gmail.com}}}
\address
{\textit{Khaldoun Al-Zoubi, Department of Mathematics and
Statistics, Jordan University of Science and Technology, P.O.Box
3030, Irbid 22110, Jordan.}}
\bigskip
{\email{\textit{kfzoubi@just.edu.jo}}}


\date{}
\begin{abstract}
Let $G$ be a group with identity $e$. Let $R$ be a $G$-graded commutative
ring with identity and $M$ a graded $R$-module. In this paper, we introduce
the concept of graded $I_{e}$-prime submodule as a generalization of a
graded prime submodule for $I=\oplus _{g\in G}I_{g}$ a fixed graded ideal of
$R$. We give a number of results concerning of these classes of graded
submodules and their homogeneous components. A proper graded submodule $N$
of $M$ is said to be a graded $I_{e}$-prime submodule of $M$ if whenever $%
r_{g}\in h(R)$ and $m_{h}\in h(M)$ with $r_{g}m_{h}\in N-I_{e}N,$ then
either $r_{g}\in (N:_{R}M)$ or $m_{h}\in N.$
\end{abstract}

\keywords{graded $I_{e}$-prime submodules, graded prime submodules, graded $I_{e}$-prime ideals. \\
$*$ Corresponding author}
 \maketitle
\section{Introduction and Preliminaries}

Throughout this paper all rings are commutative with identity and
all modules are unitary.

Graded prime submodules of graded modules over graded commutative rings,
have been introduced and studied by many authors, (see for example [1, 3-4, 6-7, 10, 16]).
Then, many generalizations of graded prime submodules were studied such as
graded primary, graded classical prime, graded weakly prime and graded 2-absorbing submodules (see for
example \cite{2, 5, 9, 16}).
Akray and Hussein in [8] introduced the concept of $I$-prime submodule over
a commutative ring as a new generalization of prime submodule.

The scope of this paper is devoted
to the theory of graded modules over graded commutative rings. Here, we introduce the concept of graded $I_{e}$-prime submodule as a new
generalization of a graded prime submodule. A number of results concerning
of these classes of graded submodules and their homogeneous components are
given. For example, we give a characterization of graded $I_{e}$-prime
submodule (see Theorem 2.15). We also study the behaviour of graded $I_{e}$%
-prime submodule under localization (see Theorem 2.16).  

First, we recall some basic properties of graded rings and modules which
will be used in the sequel. We refer to [12-15] for these basic
properties and more information on graded rings and modules.

Let $G$ be a multiplicative group with identity element $e.$ A ring $R$ is
called a graded ring (or $G$-graded ring) if there exist additive subgroups $%
R_{g}$ of $R$ indexed by the elements $g\in G$ such that $R=\oplus _{g\in
G}R_{g}$ and $R_{g}R_{h}\subseteq R_{gh}$ for all $g,h\in G$. The elements
of $R_{g}$ are called homogeneous of degree $g$ and all the homogeneous
elements are denoted by $h(R)$, i.e. $h(R)=\cup _{g\in G}R_{g}$. If $r\in R$%
, then $r$ can be written uniquely as $\sum_{g\in G}r_{g}$, where $r_{g}$ is
called a homogeneous component of $r$ in $R_{g}$. Moreover, $R_{e}$ is a
subring of $R$ and $1\in R_{e}$. Let $R=\oplus _{g\in G}R_{g}$ be a $G$%
-graded ring. An ideal $I$ of $R$ is said to be a graded ideal if $%
I=\sum_{g\in G}(I\cap R_{g}):=\sum_{g\in G}I_{g}$, (see \cite{15}. Let $R=\oplus _{g\in G}R_{g}$ be a $G$-graded ring. A left $R$-module $M$ is said to be a graded $R$-module (or $G$-graded $R$%
-module) if there exists a family of additive subgroups $\{M_{g}\}_{g\in G}$
of $M$ such that $M=\oplus _{g\in G}M_{g}$ and $R_{g}M_{h}\subseteq M_{gh}$
for all $g,h\in G$. Also if an element of $M$ belongs to $\cup _{g\in
G}M_{g}=h(M)$, then it is called a homogeneous. Note that $M_{g}$ is an $%
R_{e}$-module for every $g\in G$. Let $R=\oplus _{g\in G}R_{g}$ be a $G$%
-graded ring. A submodule $N$ of $M$ is said to be a graded submodule of $M$
if $N=\oplus _{g\in G}(N\cap M_{g}):=\oplus _{g\in G}N_{g}$. In this case, $%
N_{g}$ is called the $g$-component of $N$. Moreover, $M/N$ becomes a $G$%
-graded $R$-module with $g$-component $(M/N)_{g}:=(M_{g}+N)/N$ for $g\in G$,
(see \cite{15}).

Let $R$ be a $G$-graded ring and $S\subseteq h(R)$ be a multiplicatively
closed subset of $R$. Then the ring of fraction $S^{-1}R$ is a graded ring
which is called the graded ring of fractions. Indeed, $S^{-1}R=\oplus _{g\in
G}(S^{-1}R)_{g}$ where $(S^{-1}R)_{g}=\{r/s:r\in R,s\in S$ and $%
g=(degs)^{-1}(degr)\}$. Let $M$ be a graded module over a $G$-graded ring $R$
and $S\subseteq h(R)$ be a multiplicatively closed subset of $R$. The module
of fractions $S^{-1}M$ over a graded ring $S^{-1}R$ is a graded module which
is called the module of fractions, if $S^{-1}M=\oplus _{g\in G}(S^{-1}M)_{g}$
where $(S^{-1}M)_{g}=\{m/s:m\in M,s\in S$ and $g=(degs)^{-1}(degm)\}$. We
write $h(S^{-1}R)=\cup _{g\in G}(S^{-1}R)_{g}$ and $h(S^{-1}M)=\cup _{g\in
G}(S^{-1}M)_{g}$. For any graded submodule $N$ of $M$, the graded submodule
of $S^{-1}N$ of $S^{-1}M$ is defined by $S^{-1}N=\{\alpha \in S^{-1}M:\alpha
=m/s$ for $m\in N$ and $s\in S\}$ and $S^{-1}N\not=S^{-1}M$ if and only if $%
S\cap (N:_{R}M)=\emptyset ,$ (see \cite{15}).

The graded radical of a graded ideal $I$, denoted by $Gr(I)$, is the set of
all $r=\sum_{g\in G}r_{g}\in R$ such that for each $g\in G$ there exists $%
n_{g}\in
\mathbb{N}
$ with $r_{g}^{n_{g}}\in I$. Note that, if $x$ is a homogeneous element,
then $x\in Gr(I)$ if and only if $x^{n}\in I$ for some $n\in \mathbb{N} $, (see \cite{17}).

Let $R$ be a $G$-graded ring and $M$ be a graded $R$-module. It is shown in \cite[Lemma 2.1]{10} that if $N$ is a graded submodule of $M$, then $%
(N:_{R}M)=\{r\in R:rN\subseteq M\}$ is a graded ideal of $R$. Also, for any $%
r_{g}\in h(R)$, the graded submodule $\{m\in M:r_{g}m\in N\}$ will be
denoted by $(N:_{M}r_{g})$. The \textit{graded radical} of a graded
submodule $N$ of $M$, denoted by $Gr_{M}(N)$, is defined to be the
intersection of all graded prime submodules of $M$ containing $N$. If $N$ is
not contained in any graded prime submodule of $M$, then $Gr_{M}(N)=M$, (see \cite{16}).



 \section{RESULTS}

\begin{definition}
Let $R$ be a $G$-graded ring, $M$ a graded $R$-module, $I=\oplus _{g\in G}I_{g}$ a graded ideal of $R$, $N=\oplus _{g\in
G}N_{g}$ a graded submodule of $M$ and $g\in G.$
 \begin{enumerate}[\upshape (i)]
   \item  We say that $N_{g}$ is a $g$-$I_{e}$-prime submodule of the $R_{e}$%
-module $M_{g}$ if $N_{g}\not=M_{g}$; and whenever $r_{e}\in R_{e}$ and $%
m_{g}\in M_{g}$ with $r_{e}m_{g}\in N_{g}-I_{e}N_{g}$, implies either $%
m_{g}\in N_{g}$ or $r_{e}\in (N_{g}:_{R_{e}}M_{g}).$
   \item  We say that $N$ is a graded $I_{e}$-prime submodule of $M$ if $%
N\not=M $; and whenever $r_{h}\in h(R)$ and $m_{\lambda }\in h(M)$ with $%
r_{h}m_{\lambda }\in N-I_{e}N$, implies either $m_{\lambda }\in N$ or $%
r_{h}\in (N:_{R}M).$
 \end{enumerate}
\end{definition}
\begin{definition}
Let $R$ be a $G$-graded ring, $I=\oplus _{g\in
G}I_{g}$ and $J=\oplus _{g\in G}J_{g}$ be graded ideals of $R.$ Then $J_{e}$
is said to be an $e$-$I_{e}$-prime ideal of $R_{e},$ if $J_{e}\not=R_{e};$
and whenever $r_{e}s_{e}\in J_{e}-I_{e}J_{e},$ where $r_{e},s_{e}\in R_{e}$,
implies either $r_{e}\in J_{e}$ or $s_{e}\in J_{e}.$
\end{definition}
Let $R$ be a $G$-graded ring, $M$ a graded $R$-module and $I=\oplus _{g\in
G}I_{g}$ a graded ideal of $R.$
Recall from \cite{10} that a proper graded submodule $N=\oplus _{g\in
G}N_{g}$ of a graded $R$-module $M$ is said to be \textit{a graded prime submodule of $M$} if whenever $r_{g}m_{h}\in N$
where $r_{g}\in h(R)$ and $m_{h}\in h(M),$ then either $m_{h}\in N$ or $%
r_{g}\in (N:_{R}M).$ It is easy to see that every graded prime submodule of $%
M$ is a graded $I_{e}$-prime submodule. The following example shows that the
converse is not true in general.

\begin{example}
Let $G=%
\mathbb{Z}
_{2}$ and $R=%
\mathbb{Z}
$ be a $G$-graded ring with $R_{0}=%
\mathbb{Z}
$ and $R_{1}=\{0\}.$ Let $M=%
\mathbb{Z}
_{12}$ be a graded $R$-module with $M_{0}=%
\mathbb{Z}
_{12}$ and $M_{1}=\{0\}.$ Now, consider a graded ideal $I=4%
\mathbb{Z}
$ of $R$ and a graded submodule $N=\langle 4\rangle $ of $M,$ then $N$ is
not a graded prime submodule of $M$ since $2\cdot 2=4\in N$ but neither $%
2\in N$ nor $2\in (N:_{R}M)=4%
\mathbb{Z}
.$ However, $N$ is a graded $I_{e}$-prime submodule of $M$ since $%
N-I_{0}N=\langle 4\rangle -4%
\mathbb{Z}
\cdot \langle 4\rangle =\emptyset .$
\end{example}
Recall from \cite{10} that if $N=\oplus _{g\in G}N_{g}$ is a graded submodule of a graded $%
R $-module $M$ and $g\in G$,  then $N_{g}$ is called \textit{a }$g$\textit{-prime
submodule of an }$R_{e}$\textit{-module }$M_{g}$ if $N_{g}\not=M_{g}$; and
whenever $r_{e}m_{g}\in
N_{g} $ where $r_{e}\in R_{e}$ and $m_{g}\in M_{g}$, then either $m_{g}\in N_{g}$ or $r_{e}\in (N_{g}:_{R_{e}}M_{g})$.

\begin{theorem}
Let $R$ be a $G$-graded ring, $M$ a graded $R$%
-module, $I=\oplus _{g\in G}I_{g}$ a graded ideal of $R$, $N=\oplus _{g\in
G}N_{g}$ a graded submodule of $M$ and $g\in G.$ If $N_{g}$ is a $g$-$I_{e}$%
-prime submodule of $M_{g}$, then either $N_{g}$ is a $g$-prime submodule of
$M_{g}$ or $(N_{g}:_{R_{e}}M_{g})N_{g}\subseteq I_{e}N_{g}.$
\end{theorem}
\begin{proof}
Suppose that $N_{g}$ is a $g$-$I_{e}$-prime submodule of $M_{g}$
such that $(N_{g}:_{R_{e}}M_{g})N_{g}\not\subseteq I_{e}N_{g}.$ Now, let $%
r_{e}\in R_{e}$ and $m_{g}\in M_{g}$ with $r_{e}m_{g}\in N_{g}.$ If $%
r_{e}m_{g}\not\in I_{e}N_{g}$, then either $m_{g}\in N_{g}$ or $r_{e}\in
(N_{g}:_{R_{e}}M_{g})$ as $N_{g}$ is a $g$-$I_{e}$-prime submodule of $%
M_{g}. $ Assume that $r_{e}m_{g}\in I_{e}N_{g}$. If $r_{e}N_{g}\not\subseteq
I_{e}N_{g},$ then there exists $x_{g}\in N_{g}$ such that $%
r_{e}x_{g}\not\in I_{e}N_{g}$, so we get $r_{e}(m_{g}+x_{g})\in
N_{g}-I_{e}N_{g}$ and then either $m_{g}+x_{g}\in N_{g}$ or $r_{e}\in
(N_{g}:_{R_{e}}M_{g})$ as $N_{g}$ is a $g$-$I_{e}$-prime submodule of $%
M_{g}. $ Hence, either $m_{g}\in N_{g}$ or $r_{e}\in (N_{g}:_{R_{e}}M_{g}).$
If $(N_{g}:_{R_{e}}M_{g})m_{g}\not\subseteq I_{e}N_{g}$, there exists $%
t_{e}\in (N_{g}:_{R_{e}}M_{g})$ such that $t_{e}m_{g}\not\in I_{e}N_{g}$, so
we get $(r_{e}+t_{e})m_{g}\in N_{g}-I_{e}N_{g}$ and then either $m_{g}\in
N_{g}$ or $r_{e}+t_{e}\in (N_{g}:_{R_{e}}M_{g})$ as $N_{g}$ is a $g$-$I_{e}$%
-prime submodule of $M_{g}.$ Hence, either $m_{g}\in N_{g}$ or $r_{e}\in
(N_{g}:_{R_{e}}M_{g}).$ Now, we can assume that $r_{e}N_{g}\subseteq
I_{e}N_{g}$ and $(N_{g}:_{R_{e}}M_{g})m_{g}\subseteq I_{e}N_{g}.$ But $%
(N_{g}:_{R_{e}}M_{g})N_{g}\not\subseteq I_{e}N_{g}$, so there exist $%
s_{e}\in (N_{g}:_{R_{e}}M_{g})$ and $l_{g}\in N_{g}$ such that $%
s_{e}l_{g}\not\in I_{e}N_{g}.$ Thus $(r_{e}+s_{e})(m_{g}+l_{g})\in
N_{g}-I_{e}N_{g}$ gives either $m_{g}+l_{g}\in N_{g}$ or $r_{e}+s_{e}\in
(N_{g}:_{R_{e}}M_{g})$ as $N_{g}$ is a $g$-$I_{e}$-prime submodule of $%
M_{g}. $ Hence, either $m_{g}\in N_{g}$ or $r_{e}\in (N_{g}:_{R_{e}}M_{g}).$
Therefore, $N_{g}$ is a $g$-prime submodule of $M_{g}$.
\end{proof}
\begin{corollary}
Let $R$ be a $G$-graded ring, $M$ a graded $R$%
-module, $N=\oplus _{g\in G}N_{g}$ a graded submodule of $M$ and $g\in G.$
If $N_{g}$ is a $g$-$0$-prime submodule of $M_{g}$ such that $%
(N_{g}:_{R_{e}}M_{g})N_{g}\not=0$, then $N_{g}$ is a $g$-prime submodule of $%
M_{g}.$
\end{corollary}
\begin{proof}
Take $I=0$ in the Theorem 2.4.
\end{proof}

Recall from \cite{11} that a graded $R$-module $M$ is called a graded multiplication if for
each graded submodule $N$ of $M$, we have $N=IM$ for some graded ideal $I$
of $R$. If $N$ is graded submodule of a graded multiplication module $M$,
then $N=(N:_{R}M)M$.
Let $N$ and $K$ be two graded submodules of a graded multiplication $R$%
-module $M$ with $N=I_{1}M$ and $K=I_{2}M$ for some graded ideals $I_{1}$
and $I_{2}$ of $R$. The product of $N$ and $K$ denoted by $NK$ is defined by
$NK=I_{1}I_{2}M.$

 \begin{corollary}
Let $R$ be a $G$-graded ring, $M$ a graded $R$%
-module, $I=\oplus _{g\in G}I_{g}$ a graded ideal of $R$, $N=\oplus _{g\in
G}N_{g}$ a graded submodule of $M$ and $g\in G.$ If $M_{g}$ is a
multiplication $R_{e}$-module and $N_{g}$ is a $g$-$I_{e}$-prime submodule
which is not a $g$-prime submodule of $M_{g}$, then $N_{g}^{2}\subseteq
I_{e}N_{g}.$
\end{corollary}
\begin{proof}
Suppose that $N_{g}$ is a $g$-$I_{e}$-prime submodule of $M_{g}$
which is not a $g$-prime, so by Theorem 2.4, we get $%
(N_{g}:_{R_{e}}M_{g})N_{g}\subseteq I_{e}N_{g}.$ Now, as $M_{g}$ is a
multiplication $R_{e}$-module we get $%
N_{g}^{2}=(N_{g}:_{R_{e}}M_{g})^{2}M_{g}\subseteq
(N_{g}:_{R_{e}}M_{g})N_{g}\subseteq I_{e}N_{g}.$ Therefore, $%
N_{g}^{2}\subseteq I_{e}N_{g}.$
\end{proof}


\begin{theorem}
Let $R$ be a $G$-graded ring, $M$ a graded $R$%
-module, $I=\oplus _{g\in G}I_{g}$ a graded ideal of $R$, $N=\oplus _{g\in
G}N_{g}$ a graded submodule of $M$ and $g\in G.$ Then the following
statements are equivalent:
\begin{enumerate}[\upshape (i)]
  \item $N_{g}$ is a $g$-$I_{e}$-prime submodule of $M_{g}$.
  \item For $r_{e}\in R_{e}\backslash (N_{g}:_{R_{e}}M_{g})$, $%
(N_{g}:_{M_{g}}r_{e})=N_{g}\cup (I_{e}N_{g}:_{M_{g}}r_{e}).$
  \item For $r_{e}\in R_{e}\backslash (N_{g}:_{R_{e}}M_{g})$, either $%
(N_{g}:_{M_{g}}r_{e})=N_{g}$ or $(N_{g}:_{M_{g}}r_{e})=(I_{e}N_{g}:_{M_{g}}r_{e}).$
\end{enumerate}
\end{theorem}
\begin{proof}
$(i)\Rightarrow (ii)$ Suppose that $N_{g}$ is a $g$-$I_{e}$-prime
submodule of $M_{g}$ and $r_{e}\in R_{e}\backslash (N_{g}:_{R_{e}}M_{g})$.
Let $m_{g}\in (N_{g}:_{M_{g}}r_{e}),$ so $r_{e}m_{g}\in N_{g}.$ If $%
r_{e}m_{g}\not\in I_{e}N_{g},$ then $m_{g}\in N_{g}.$ Now, if $r_{e}m_{g}\in
I_{e}N_{g},$ then $m_{g}\in (I_{e}N_{g}:_{M_{g}}r_{e}).$ Hence, $%
(N_{g}:_{M_{g}}r_{e})=N_{g}\cup (I_{e}N_{g}:_{M_{g}}r_{e}).$

$(ii)\Rightarrow (iii)$ If a submodule is a union of two submodules, then it
is equal one of them.

$(iii)\Rightarrow (i)$ Let $r_{e}\in R_{e}$ and $m_{g}\in M_{g}$ such that $%
r_{e}m_{g}\in N_{g}-I_{e}N_{g}$ and $r_{e}\not\in (N_{g}:_{R_{e}}M_{g})$, so
$m_{g}\in (N_{g}:_{M_{g}}r_{e}).$ Now, by $(iii)$ we get either $%
(N_{g}:_{M_{g}}r_{e})=N_{g}$ or $%
(N_{g}:_{M_{g}}r_{e})=(I_{e}N_{g}:_{M_{g}}r_{e}).$ But $r_{e}m_{g}\not\in
I_{e}N_{g}$, so $m_{g}\not\in (I_{e}N_{g}:_{M_{g}}r_{e}).$ Hence, $%
(N_{g}:_{M_{g}}r_{e})=N_{g}$ and then $m_{g}\in N_{g}.$ Therefore, $N_{g}$
is a $g$-$I_{e}$-prime submodule of $M_{g}$.
\end{proof}
Recall from \cite{9} that a proper graded submodule $N=\oplus _{g\in
G}N_{g}$ of a graded $R$-module $M$ is said to be \textit{a graded weakly prime submodule of }$M$ if whenever $0\neq
r_{g}m_{h}\in N$ where $r_{g}\in h(R)$ and $m_{h}\in h(M),$ then either $%
m_{h}\in N$ or $r_{g}\in (N:_{R}M).$
\begin{theorem}
Let $R$ be a $G$-graded ring, $M$ a graded $R$-module, $I=\oplus _{g\in G}I_{g}$ a graded ideal of $R$ and $N$ a proper
graded submodule of $M.$ Then the following statements are equivalent:
\begin{enumerate}[\upshape (i)]
  \item $N$ is a graded $I_{e}$-prime submodule of $M$.
  \item $N/I_{e}N$ is a graded weakly prime submodule of $M/I_{e}N.$
\end{enumerate}
\end{theorem}
\begin{proof}
$(i)\Rightarrow (ii)$ Suppose that $N$ is a graded $I_{e}$-prime
submodule of $M$. Let $r_{g}\in h(R)$ and $(m_{h}+I_{e}N)\in h(M/I_{e}N)$
with $0_{M/I_{e}N}\not=(r_{g}m_{h}+I_{e}N)\in N/I_{e}N$, this yields that $%
r_{g}m_{h}\in N-I_{e}N.$ Hence, either $m_{h}\in N$ or $r_{g}M\subseteq N$
as $N$ is a graded $I_{e}$-prime submodule of $M$. Then either $%
(m_{h}+I_{e}N)\in N/I_{e}N$ or $r_{g}(M/I_{e}N)\subseteq N/I_{e}N.$
Therefore, $N/I_{e}N$ is a graded weakly prime submodule of $M/I_{e}N.$

$(i)\Rightarrow (ii)$ Suppose that $N/I_{e}N$ is a graded weakly
prime submodule of $M/I_{e}N.$ Let $r_{g}\in h(R)$ and $m_{h}\in h(M)$ such
that $r_{g}m_{h}\in N-I_{e}N.$ This follows that $\
0_{M/I_{e}N}\not=(r_{g}m_{h}+I_{e}N)=r_{g}(m_{h}+I_{e}N)\in N/I_{e}N.$ Thus,
either $r_{g}\in (N/I_{e}N:_{R}M/I_{e}N)$ or $(m_{h}+I_{e}N)\in N/I_{e}N$
and then either $r_{g}\in (N:_{R}M)$ or $m_{h}\in N.$ Therefore, $N$ is a
graded $I_{e}$-prime submodule of $M$.
\end{proof}
\begin{lemma}
Let $R$ be a $G$-graded ring, $M$ a graded $R$%
-module and $I=\oplus _{g\in G}I_{g}$ be a graded ideal of $R$ and $g\in G.$
If $M_{g}$ is a multiplication $R_{e}$-module and $N_{g}$ is a $g$-$I_{e}$%
-prime submodule of $M_{g}$ with $(N_{g}:_{R_{e}}M_{g})\subseteq I_{e}$,
then $Gr((I_{e}N_{g}:_{R_{e}}M_{g}))N_{g}=I_{e}N_{g}.$
\end{lemma}
\begin{proof}
Suppose that $N_{g}$ is a $g$-$I_{e}$-prime submodule of $M_{g}$
with $(N_{g}:_{R_{e}}M_{g})\subseteq I_{e}$. Clearly, $I_{e}N_{g}\subseteq
(I_{e}N_{g}:_{R_{e}}M_{g})N_{g}\subseteq
Gr((I_{e}N_{g}:_{R_{e}}M_{g}))N_{g}. $ Now, let $r_{e}\in
Gr((I_{e}N_{g}:_{R_{e}}M_{g})).$ If $r_{e}\in I_{e}$, then $%
r_{e}N_{g}\subseteq I_{e}N_{g}.$ So we can assume that $r_{e}\not\in I_{e},$
which follows that $r_{e}\not\in (N_{g}:_{R_{e}}M_{g}),$ then by Theorem 2.7, either $(N_{g}:_{M_{g}}r_{e})=N_{g}$ or $%
(N_{g}:_{M_{g}}r_{e})=(I_{e}N_{g}:_{M_{g}}r_{e}).$ If $%
(N_{g}:_{M_{g}}r_{e})=(I_{e}N_{g}:_{M_{g}}r_{e}),$ then $%
r_{e}N_{g}\subseteq
r_{e}(N_{g}:_{M_{g}}r_{e})=r_{e}(I_{e}N_{g}:_{M_{g}}r_{e})\subseteq
I_{e}N_{g}.$ On the other hand, if $(N_{g}:_{M_{g}}r_{e})=N_{g}$, let $n\geq
2$ be the smallest integer such that $r_{e}^{n}\in
(I_{e}N_{g}:_{R_{e}}M_{g}).$ Since $I_{e}N_{g}\subseteq N_{g}$, $%
r_{e}^{n}M_{g}=r_{e}(r_{e}^{n-1}M_{g})\subseteq N_{g}$ which yields that $%
(r_{e}^{n-1}M_{g})\subseteq N_{g}.$ Thus, after a finite number of steps we
get $r_{e}M_{g}\subseteq N_{g}$ which is a contradiction. Hence, $%
Gr((I_{e}N_{g}:_{R_{e}}M_{g}))N_{g}\subseteq I_{e}N_{g}.$ Therefore, $%
Gr((I_{e}N_{g}:_{R_{e}}M_{g}))N_{g}=I_{e}N_{g}.$
\end{proof}
\begin{theorem}
Let $R$ be a $G$-graded ring, $M$ a graded $R$%
-module, $I=\oplus _{g\in G}I_{g}$ a graded ideal of $R$, $N=\oplus
_{g\in G}N_{g}$ a graded submodule of $M$ and $g\in G$ with $%
(I_{e}N_{g}:_{R_{e}}M_{g})=I_{e}(N_{g}:_{R_{e}}M_{g}).$ If $N_{g}$ is a $g$-$%
I_{e}$-prime submodule of $M_{g}$, then $(N_{g}:_{R_{e}}M_{g})$ is an $e$-$%
I_{e}$-prime ideal of $R_{e}.$
\end{theorem}
\begin{proof}
Suppose that $N_{g}$ is a $g$-$I_{e}$-prime submodule of $M_{g}$. Now, let $r_{e},s_{e}\in R_{e}$ such that $r_{e}s_{e}\in
(N_{g}:_{R_{e}}M_{g})-I_{e}(N_{g}:_{R_{e}}M_{g})$ and $r_{e}\not\in
(N_{g}:_{R_{e}}M_{g}).$ By Theorem 2.7 since $r_{e}\not\in
(N_{g}:_{R_{e}}M_{g}),$ we get either $(N_{g}:_{M_{g}}r_{e})=N_{g}$ or $%
(N_{g}:_{M_{g}}r_{e})=(I_{e}N_{g}:_{M_{g}}r_{e}).$ If $r_{e}s_{e}M_{g}%
\subseteq I_{e}N_{g}$, then $r_{e}s_{e}\in
(I_{e}N_{g}:_{R_{e}}M_{g})=I_{e}(N_{g}:_{R_{e}}M_{g}),$ a contradiction.
Hence, $r_{e}s_{e}M_{g}\not\subseteq I_{e}N_{g}.$ Since $s_{e}M_{g}\subseteq
(N_{g}:_{M_{g}}r_{e})$ and $s_{e}M_{g}\not\subseteq
(I_{e}N_{g}:_{M_{g}}r_{e}),$ we get $s_{e}M_{g}\subseteq N_{g}.$ So $%
s_{e}\in (N_{g}:_{R_{e}}M_{g}).$
\end{proof}
\begin{theorem}
Let $R$ be a $G$-graded
ring, $I=\oplus _{g\in G}I_{g}$ a graded ideal of $R$ and $J=\oplus _{g\in
G}J_{g}$ a proper graded ideal of $R$. Then the following statements are
equivalent:
\begin{enumerate}[\upshape (i)]
  \item $J_{e}$ is an $e$-$I_{e}$-prime ideal of $R_{e}$.
  \item For $r_{e}\in R_{e}-J_{e}$, $(J_{e}:_{R_{e}}r_{e})=J_{e}\cup
(I_{e}J_{e}:_{R_{e}}r_{e}).$
  \item For $r_{e}\in R_{e}-J_{e}$, $(J_{e}:_{R_{e}}r_{e})=J_{e}$ or $%
(J_{e}:_{R_{e}}r_{e})=(I_{e}J_{e}:_{R_{e}}r_{e}).$

 \item For any two graded ideals $K=\oplus _{g\in G}K_{g}$ and $L=\oplus
_{h\in G}L_{h}$ of $R$ with $K_{e}L_{e}\subseteq J_{e}$ and $%
K_{e}L_{e}\not\subseteq I_{e}J_{e},$ implies either $K_{e}\subseteq J_{e}$
or $L_{e}\subseteq J_{e}$.
\end{enumerate}
\end{theorem}
\begin{proof}
$(i)\Rightarrow (ii)$ Suppose that $J_{e}$ is an $e$-$I_{e}$-prime
ideal of $R_{e}$ and $r_{e}\in R_{e}-J_{e}$. It is easy to see that $%
J_{e}\cup (I_{e}J_{e}:_{R_{e}}r_{e})\subseteq (J_{e}:_{R_{e}}r_{e}).$ Now,
let $s_{e}\in (J_{e}:_{R_{e}}r_{e})$, so $r_{e}s_{e}\in J_{e}$. If $%
r_{e}s_{e}\in J_{e}-I_{e}J_{e}$, then $s_{e}\in J_{e}$. If $r_{e}s_{e}\in
I_{e}J_{e}$, then $s_{e}\in (I_{e}J_{e}:_{R_{e}}r_{e})$, which follows that $%
(J_{e}:_{R_{e}}r_{e})\subseteq J_{e}\cup (I_{e}J_{e}:_{R_{e}}r_{e})$ and
hence $(J_{e}:_{R_{e}}r_{e})=J_{e}\cup (I_{e}J_{e}:_{R_{e}}r_{e})$.

 $(ii)\Rightarrow (iii)$ Note that if a ideal is a union of two
ideals, then it is equal to one of them.

$(iii)\Rightarrow (iv)$ Let $K=\oplus _{g\in G}K_{g}$ and $L=\oplus _{h\in
G}L_{h}$ be two graded ideals of $R$ such that $K_{e}L_{e}\subseteq J_{e},$ $%
K_{e}L_{e}\not\subseteq I_{e}J_{e}$ and neither $K_{e}\subseteq J_{e}$ nor $%
L_{e}\subseteq J_{e}$. Let $k_{e}\in K_{e}$, if $k_{e}\not\in J_{e}$, then $%
k_{e}L_{e}\subseteq J_{e}$ gives $L_{e}\subseteq (J_{e}:_{R_{e}}k_{e})$.
Now, by $(iii)$ since $k_{e}\in R_{e}-J_{e}$, $(J_{e}:_{R_{e}}k_{e})=J_{e}$
or $(J_{e}:_{R_{e}}k_{e})=(I_{e}J_{e}:_{R_{e}}k_{e}).$ But $%
L_{e}\not\subseteq J_{e}$, so $L_{e}\subseteq (I_{e}J_{e}:_{R_{e}}k_{e})$.
Hence, $k_{e}L_{e}\subseteq I_{e}J_{e}$. Now, if $k_{e}\in J_{e}$, then
since $K_{e}\not\subseteq J_{e}$, there exists $k_{e}^{^{\prime }}\in
K_{e}-J_{e}$, so we have $(k_{e}+k_{e}^{^{\prime }})L_{e}\subseteq J_{e}$
and by using the first case we get $k_{e}^{^{\prime }}L_{e}\subseteq
I_{e}J_{e}.$ Since $k_{e}+k_{e}^{^{\prime }}\in R_{e}-J_{e}$ and $%
L_{e}\not\subseteq J_{e},$ $L_{e}\subseteq
(I_{e}J_{e}:_{R_{e}}k_{e}+k_{e}^{^{\prime }})$ and then $(k_{e}+k_{e}^{%
\prime })L_{e}\subseteq I_{e}J_{e}.$ Hence $k_{e}L_{e}\subseteq I_{e}J_{e}$
since $k_{e}^{^{\prime }}L_{e}\subseteq I_{e}J_{e}.$ Thus, $%
K_{e}L_{e}\subseteq I_{e}J_{e},$ a contradiction.

$(iv)\Rightarrow (i)$ Let $r_{e},s_{e}\in R_{e}$ such that $r_{e}s_{e}\in
J_{e}-I_{e}J_{e}$. Then $K=(r_{e})$ and $L=(s_{e})$ are graded ideals of $%
R $ generated by $r_{e}$ and $s_{e},$ respectively. Now, $%
K_{e}L_{e}\subseteq J_{e}$ and $K_{e}L_{e}\not\subseteq I_{e}J_{e}$. So
either $K_{e}\subseteq J_{e}$ or $L_{e}\subseteq J_{e}$ and hence $r_{e}\in
J_{e}$ or $s_{e}\in J_{e}$.
\end{proof}

\begin{theorem}
Let $R$ be a $G$-graded ring, $M$ a graded $R$%
-module, $I=\oplus _{g\in G}I_{g}$ a graded ideal of $R$, $N=\oplus _{g\in
G}N_{g}$ a proper graded submodule of $M$ and $g\in G$ such that $%
(I_{e}N_{g}:_{R_{e}}M_{g})=I_{e}(N_{g}:_{R_{e}}M_{g}).$ If $N_{g}$ is a $g$-$I_{e}$-prime submodule of $M_{g}$ and $M_{g}$ is a
multiplication $R_{e}$-module, then for any graded submodules $K=\oplus
_{h\in G}K_{h}$ and $L=\oplus _{h\in G}L_{h}$ with $K_{g}L_{g}\subseteq
N_{g} $ and $K_{g}L_{g}\not\subseteq I_{e}N_{g}$, implies either $%
K_{g}\subseteq N_{g}$ or $L_{g}\subseteq N_{g}.$
\end{theorem}
\begin{proof}
Suppose that $N_{g}$ is a $g$-$I_{e}$-prime submodule of $M_{g}$ so
by Theorem 2.10, we get $(N_{g}:_{R_{e}}M_{g})$ is an $e$-$I_{e}$%
-prime ideal of $R_{e}.$ Now, let $K=\oplus _{h\in G}K_{h}$ and $L=\oplus
_{h\in G}L_{h}$ be two graded submodules of $M$ such that $%
K_{g}L_{g}\subseteq N_{g}$ and $K_{g}L_{g}\not\subseteq I_{e}N_{g}$ and
neither $K_{g}\subseteq N_{g}$ nor $L_{g}\subseteq N_{g}.$ Hence, $%
K_{g}L_{g}=(K_{g}:_{R_{e}}M_{g})(L_{g}:_{R_{e}}M_{g})M_{g}\subseteq N_{g}$
and then $(K_{g}:_{R_{e}}M_{g})(L_{g}:_{R_{e}}M_{g})\subseteq
(N_{g}:_{R_{e}}M_{g}).$ Now, $K_{g}=(K_{g}:_{R_{e}}M_{g})M_{g}\not\subseteq
N_{g}$ gives $(K_{g}:_{R_{e}}M_{g})\not\subseteq (N_{g}:_{R_{e}}M_{g}),$ and
$L_{g}=(L_{g}:_{R_{e}}M_{g})M_{g}\not\subseteq N_{g}$ gives $%
(L_{g}:_{R_{e}}M_{g})\not\subseteq (N_{g}:_{R_{e}}M_{g}).$ So by Theorem 2.11, we get $(K_{g}:_{R_{e}}M_{g})(L_{g}:_{R_{e}}M_{g})\subseteq
I_{e}(N_{g}:_{R_{e}}M_{g})=(I_{e}N_{g}:_{R_{e}}M_{g})$ which yields that $%
K_{g}L_{g}=(K_{g}:_{R_{e}}M_{g})(L_{g}:_{R_{e}}M_{g})M_{g}\subseteq
I_{e}N_{g}$, a contradiction. Therefore, either $K_{g}\subseteq N_{g}$ or $%
L_{g}\subseteq N_{g}.$
\end{proof}

Suppose $M_{g}$ is a multiplication $R_{e}$-module and $m_{1_{g}}$, $%
m_{2_{g}}\in M_{g}$. Then we can define the product of $m_{1_{g}}$ and $m_{2_{g}}$ as $m_{1_{g}}m_{2_{g}}$ $%
=R_{e}m_{1_{g}}\
R_{e}m_{2_{g}}=(R_{e}m_{1_{g}}:_{R_{e}}M_{g})(R_{e}m_{2_{g}}:M_{g})M_{g}.$ Thus we have the following corollary.
\begin{corollary}
Let $R$ be a $G$-graded ring, $M$ a graded $%
R $-module, $I=\oplus _{g\in G}I_{g}$ a graded ideal of $R$, $N=\oplus
_{g\in G}N_{g}$ a proper graded submodule of $M$ and $g\in G$ such that $%
(I_{e}N_{g}:_{R_{e}}M_{g})=I_{e}(N_{g}:_{R_{e}}M_{g})$. If $M_{g}$ is a
multiplication $R_{e}$-module and $N_{g}$ is a $g$-$I_{e}$-prime submodule
of $M_{g}$, then for any $m_{1_{g}},m_{2_{g}}\in M_{g}$ with $%
m_{1_{g}}m_{2_{g}}\in N_{g}-I_{e}N_{g}$, implies either $m_{1_{g}}\in N_{g}$
or $m_{2_{g}}\in N_{g}.$
\end{corollary}
\begin{proof}
Let $m_{1_{g}},m_{2_{g}}\in M_{g}$ such that $m_{1_{g}}m_{2_{g}}\in
N_{g}-I_{e}N_{g}.$ Then $K=(m_{1_{g}})$ and $L=(m_{2_{g}})$ are graded
submodules of $M$ generated by $m_{1_{g}}$ and $m_{2_{g}}$, respectively.
Since $K_{g}L_{g}\subseteq N_{g}$ and $K_{g}L_{g}\not\subseteq I_{e}N_{g}$,
by Theorem 2.12, we get either $m_{1_{g}}\in N_{g}$ or $%
m_{2_{g}}\in N_{g}.$
\end{proof}


\begin{theorem}
Let $R$ be a $G$-graded ring, $M_{1}$ and $M_{2}$ be
two graded $R$-modules, $I=\oplus _{g\in G}I_{g}$ a graded ideal of $R$ and $%
N_{1}$ and $N_{2}$ be two graded submodules of $M_{1}$ and $M_{2}$,
respectively. Then:
\begin{enumerate}[\upshape (i)]
  \item If $N_{1}$ is a graded $I_{e}$-prime submodule of $M_{1},$ then $%
N_{1}\times M_{2}$ is a graded $I_{e}$-prime submodule of $M_{1}\times
M_{2}. $
  \item If $N_{2}$ is a graded $I_{e}$-prime submodule of $M_{2},$ then $%
M_{1}\times N_{2}$ is a graded $I_{e}$-prime submodule of $M_{1}\times
M_{2}. $
\end{enumerate}
\end{theorem}
\begin{proof}
$(i)$ Suppose that $N_{1}$ is a graded $I_{e}$-prime submodule of $%
M_{1}.$ Now, let $r_{g}\in h(R)$ and $(m_{1_{h}},m_{2_{h}})\in h(M_{1}\times
M_{2})$ such that $%
r_{g}(m_{1_{h}},m_{2_{h}})=(r_{g}m_{1_{h}},r_{g}m_{2_{h}})\in (N_{1}\times
M_{2})-I_{e}(N_{1}\times M_{2})=(N_{1}-I_{e}N_{1})\times (M_{2}-I_{e}M_{2}),$
which follows that $r_{g}m_{1_{h}}\in N_{1}-I_{e}N_{1}.$ Hence, either $%
m_{1_{h}}\in N_{1}$ or $r_{g}M_{1}\subseteq N_{1}$ and then either $%
(m_{1_{h}},m_{2_{h}})\in N_{1}\times M_{2}$ or $r_{g}(M_{1}\times
M_{2})\subseteq N_{1}\times M_{2}.$ Therefore, $N_{1}\times M_{2}$ is a
graded $I_{e}$-prime submodule of $M_{1}\times M_{2}.$

$(ii)$ The proof is similar to $(i).$
\end{proof}


\begin{theorem}
Let $R$ be a $G$-graded ring, $M$ a graded $%
R $-module, $I=\oplus _{g\in G}I_{g}$ a graded ideal of $R$ and $N$ a proper
graded submodule of $M.$ Let $K=\oplus _{h\in G}K_{h}$ be a graded submodule
of $M$. Then the following statements are equivalent:
\begin{enumerate}[\upshape (i)]
  \item $N$ is a graded $I_{e}$-prime submodule of $M.$
  \item For any $r_{g}\in h(R)$ and $h\in G$ with $r_{g}K_{h}\subseteq N$ and
$r_{g}K_{h}\not\subseteq I_{e}N,$ implies either $K_{h}\subseteq N$ or $%
r_{g}\in (N:_{R}M).$

\end{enumerate}
\end{theorem}
\begin{proof}
$(i)\Rightarrow (ii)$ Suppose that $N$ is a graded $I_{e}$-prime
submodule of $M.$ Now, let $r_{g}\in h(R)$ and $h\in G$ such that $%
r_{g}K_{h}\subseteq N$, $r_{g}K_{h}\not\subseteq I_{e}N$ and $r_{g}\not\in
(N:_{R}M).$ Let $k_{h}\in K_{h}$, if $r_{g}k_{h}\not\in I_{e}N,$ then $r_{g}k_{h}\in
N-I_{e}N$. So $k_{h}\in N$ as $N$ \ is a graded $I_{e}$-prime submodule of $%
M $ $.$ Now, if $r_{g}k_{h}\in I_{e}N,$ since $r_{g}K_{h}\not\subseteq
I_{e}N,$ there exists $k_{h}^{^{\prime }}\in K_{h}$ such that $%
r_{g}k_{h}^{^{\prime }}\not\in I_{e}N,$ but $r_{g}k_{h}^{^{\prime }}\in
N-I_{e}N$ and $r_{g}\not\in (N:_{R}M),$ so $k_{h}^{^{\prime }}\in N.$\
Hence, we get $r_{g}(k_{h}+k_{h}^{^{\prime }})\in N-I_{e}N$, which yields
that $k_{h}+k_{h}^{^{\prime }}\in N$ and then $k_{h}\in N.$ Therefore, $%
K_{h}\subseteq N.$

$(ii)\Rightarrow (i)$ Let $r_{g}\in h(R)$ and $m_{h}\in h(M)$ such that $%
r_{g}m_{h}\in N-I_{e}N.$ Then $K=(m_{h})$ is a graded submodule generated by
$m_{h}$. Since $r_{g}K_{h}\subseteq N$ and $r_{g}K_{h}\not\subseteq I_{e}N,$
by $(ii),$ we get either $K_{h}\subseteq N$ or $r_{g}\in (N:_{R}M)$ and then
either $m_{h}\in N$ or $r_{g}\in (N:_{R}M).$ Therefore, $N$ is a graded $%
I_{e}$-prime submodule of $M.$
\end{proof}

Recall that a graded zero-divisor on a graded $R$-module $M$ is an element $%
r_{g}\in h(R)$ for which there exists $m_{h}\in h(M)$ such that $m_{h}\not=0$
but $r_{g}m_{h}=0$. The set of all graded zero-divisors on $M$ is denoted by
$G$-$Zdv_{R}(M)$, see \cite{4}.

The following result studies the behavior of graded $I_{e}$-prime submodules
under localization.
\begin{theorem}
Let $R$ be a $G$-graded ring, $M$ a graded $R$%
-module, $S\subseteq h(R)$ be a multiplicatively closed subset of $R$ and $%
I=\oplus _{h\in G}I_{h}$ a graded ideal of $R$.
\begin{enumerate}[\upshape (i)]
  \item If $N$ is a graded $I_{e}$-prime submodule of $M$ with $(N:_{R}M)\cap
S=\emptyset $, then $S^{-1}N$ is a graded $I_{e}$-prime submodule of $%
S^{-1}M $.
  \item If $S^{-1}N$ is a graded $I_{e}$-prime submodule of $S^{-1}M$ with $%
S\cap G$-$Zdv_{R}(M/N)=\emptyset $, then $N$ is a graded $I_{e}$-prime
submodule of $M$.
\end{enumerate}
\end{theorem}
\begin{proof}
$(i)$ Since $(N:_{R}M)\cap S=\emptyset ,$ $S^{-1}N$ is a proper
graded submodule of $S^{-1}M.$\ Let $\frac{r_{g}}{s_{1}}\in h(S^{-1}R)$ and $%
\frac{m_{h}}{s_{2}}\in h(S^{-1}M)$ such that $\frac{r_{g}}{s_{1}}\frac{m_{h}%
}{s_{2}}\in S^{-1}N-I_{e}S^{-1}N.$ Then there exists $t\in S$ such that $%
tr_{g}m_{h}\in N-I_{e}N$ which yields that either $tm_{h}\in N$ or $r_{g}\in
(N:_{R}M)$ as $N$ is a graded $I_{e}$-prime submodule of $M.$ Hence either $%
\frac{m_{h}}{s_{2}}=\frac{tm_{h}}{ts_{2}}\in S^{-1}N$\ or $\frac{r_{g}}{s_{1}%
}\in S^{-1}(N:_{R}M)=(S^{-1}N:_{S^{-1}R}S^{-1}M).$ Therefore, $S^{-1}N$ is a
graded $I_{e}$-prime submodule of $S^{-1}M$.

$(ii)$ Let $r_{g}\in h(R)$ and $m_{h}\in h(M)$ such that $r_{g}m_{h}\in
N-I_{e}N$. Then $\frac{r_{g}}{1}\frac{m_{h}}{1}\in S^{-1}N-I_{e}S^{-1}N$.
Since $S^{-1}N$ is a graded $I_{e}$-prime submodule of $S^{-1}M$, either $%
\frac{m_{h}}{1}\in S^{-1}N$ or $\frac{r_{g}}{1}\in
(S^{-1}N:_{S^{-1}R}S^{-1}M)$. If $\frac{m_{h}}{1}\in S^{-1}N$, then there
exists $t\in S$ such that $sm_{h}\in N$. Which yields that $m_{h}\in N$
since $S\cap G$-$Zdv_{R}(M/N)=\emptyset $. Now, if $\frac{r_{g}}{1}\in
(S^{-1}N:_{S^{-1}R}S^{-1}M)=S^{-1}(N:_{R}M)$, then there exists $s\in S$
such that $sr_{g}M\subseteq N$ and hence $r_{g}\in (N:_{R}M)$ since $S\cap G$%
-$Zdv_{R}(M/N)=\emptyset $. Therefore, $N$ is a graded $I_{e}$-prime
submodule of $M$.
\end{proof}


\bigskip\bigskip\bigskip\bigskip


\begin{thebibliography} {10}
\bibitem{1} K. Al-Zoubi, Some properties of graded 2-prime submodules, Asian-European
Journal of Mathematics, 8 (2), (2015) 1550016-1-- 1550016-5

\bibitem{2}K. Al-Zoubi and R. Abu-Dawwas, On graded 2-absorbing and weakly
graded 2-absorbing submodules, J. Math. Sci. Adv. Appl., 28 (2014), 45--60

\bibitem{3} K. Al-Zoubi and R. Abu-Dawwas, On graded quasi-prime submodules, Kyungpook
Math. J., 55 (2) (2015), 259-266.

\bibitem{4} K. Al-Zoubi and A. Al-Qderat, Some properties of graded comultiplication
modules, Open Mathematics, 15( 2017),187-192.

\bibitem{5} K. Al-Zoubi, M. Jaradat and R. Abu-Dawwas, On graded classical prime and
graded prime submodules, Bull. Iranian Math. Soc., 41 (1) (2015), 217-225.

\bibitem{6} K. Al-Zoubi and J. Paseka, On Graded Coprimely Packed Modules, Adv. Stud.
Contemp. Math. (Kyungshang), 29(2) (2019), 271 -- 279.

\bibitem{7} K. Al-Zoubi and F. Qarqaz, An Intersection condition for graded prime
submodules in Gr-multiplication modules, Math. Reports, 20 (3), 2018,
329-336.

\bibitem{8}I. Akray and H. S. Hussein, I-prime submodules, Acta. Math.
Academic Paedagogicae Nyiregyhaziensis, 33 ( 2017),165-173.

\bibitem{9} S.E Atani, On graded weakly prime submodules, Int. Math. Forum, 1 (2)
2006, 61-66.

\bibitem{10} S.E. Atani, On graded prime submodules, Chiang Mai J. Sci., 33 (1)
(2006), 3-7.

\bibitem{11} J. Escoriza and B. Torrecillas,  Multiplication Objects in Commutative
Grothendieck Categories, Comm. in Algebra,  26 (6) (1998), 1867-1883.

\bibitem{12}R. Hazrat, Graded Rings and Graded Grothendieck Groups, Cambridge
University Press, Cambridge, 2016.

\bibitem{13} C. Nastasescu and F. Van Oystaeyen, Graded and filtered rings and
modules, Lecture notes in mathematics 758, Berlin-New York: Springer-Verlag,
1982.

\bibitem{14} C. Nastasescu, F. Van Oystaeyen, Graded Ring Theory, Mathematical Library
28, North Holand, Amsterdam, 1982.

\bibitem{15} C. Nastasescu and F. Van Oystaeyen, Methods of Graded Rings, LNM 1836.
Berlin-Heidelberg: Springer-Verlag, 2004.

\bibitem{16} K. H. Oral, U. Tekir and A. G. Agargun, On graded prime and primary
submodules, Turk. J. Math., 35 (2011), 159--167.

\bibitem{17} M. Refai and K. Al-Zoubi, On graded primary ideals, Turk. J. Math. 28 (3) (2004), 217-229.





\end{thebibliography}
\end{document}